\documentclass[final,a4paper,11pt]{article}
\usepackage{graphicx}
\usepackage{mathptmx}
\usepackage{amssymb,amsmath, amsfonts,amsthm}
\usepackage{a4wide}
\usepackage{color}

\usepackage{hyperref}

\newcommand{\email}[1]{{\tt #1}}
\newcommand{\R}{\mathbb{R}}
\newcommand{\oR}{\overline{\R}}

\newcommand{\norm}[1]{\|#1\|}

\newcommand{\Norm}[1]{\left\|#1\right\|}
\newcommand{\dist}[1]{{\rm dist}(#1)}

\newcommand{\mv}{\,\mid\,}

\newcommand{\B}{{\cal B}}

\newcommand{\I}{{\cal I}}

\newcommand{\Sp}{{\mathcal S}}

\newcommand{\Z}{{\cal Z}}

\newcommand{\setto}[1]{\mathop{\rightarrow}\limits^#1}
\newcommand{\attto}[1]{\mathop{\rightarrow}\limits_#1}
\newcommand{\longsetto}[1]{\mathop{\longrightarrow}\limits^{#1}}

\newcommand{\attconv}[1]{\mathop{\longrightarrow}\limits_{#1}^{\gph \partial #1}}
\newcommand{\skalp}[1]{\langle #1\rangle}

\newcommand{\xb}{\bar x}
\newcommand{\yb}{\bar y}
\newcommand{\zb}{\bar z}
\newcommand{\ub}{\bar u}

\newcommand{\AT}[2]{{\textstyle{#1\atop#2}}}
\newcommand{\xba}{{\bar x^\ast}}

\newcommand{\oo}{o}
\newcommand{\OO}{{\cal O}}

\newcommand{\bd}{{\rm bd\,}}
\newcommand{\co}{{\rm conv\,}}
\newcommand{\gph}{\mathrm{gph}\,}

\newcommand{\dom}{\mathrm{dom}\,}

\newcommand{\tto}{\rightrightarrows}
\newcommand{\onabla}{\overline\nabla}
\newcommand{\Limsup}{\mathop{{\rm Lim}\,{\rm sup}}}

\newcommand{\ssstar}{semismooth$^{*}$ }

\newcommand{\rge}{{\rm rge\;}}

\newcommand{\Ta}{\mathcal{T}^\varphi_\epsilon}

\newlength{\myAlgBox}
\newtheorem{theorem}{Theorem}[section]
\newtheorem{proposition}[theorem]{Proposition}
\newtheorem{remark}[theorem]{Remark}
\newtheorem{lemma}[theorem]{Lemma}
\newtheorem{corollary}[theorem]{Corollary}
\newtheorem{definition}[theorem]{Definition}

\title{On strict proto-differentiability of set-valued mappings}
\author{Helmut Gfrerer\thanks{Johann Radon Institute for  Computational and Applied Mathematics (RICAM), Altenbergerstr. 69, A-4040 Linz, Austria; \email{helmut.gfrerer@ricam.oeaw.ac.at}}
}

\date{}

\begin{document}
\maketitle

{\footnotesize
\noindent{\bf Abstract.} We will show that a multifunction is strictly proto-differentiable at a point of its graph if and only if it is graphically strictly differentiable, i.e., the graph of the multifunction locally coincides, up to a change of coordinates, with the graph of a single-valued mapping, which is strictly differentiable at the transformed reference point. This result allows point-based characterizations of strict proto-differentiability in terms of various generalized derivatives. Further we will prove that under strict proto-differentiability the properties of strong metric regularity, metric regularity and strong metric subregularity are equivalent. Finally, under strict proto-differentiability of the subgradient mapping, we provide a novel second-order relation between function values and subgradients for prox-regular functions which constitutes a nonsmooth extension of the trapezoidal rule of numerical integration.
}

\section{Introduction}

The notion of strict proto-differentiability was introduced by Poliquin and Rockafellar \cite{PolRo96} in their study of the subgradient mapping of prox-regular functions. It was shown in \cite{PolRo96} that for a prox-regular and subdifferentially continuous function $\varphi:\R^n\to(-\infty,\infty]$ the graph of the subgradient mapping $\partial \varphi$ locally coincides, up to a linear transformation of the coordinates, with the graph of the proximal mapping which is a  single-valued  and locally Lipschitzian mapping in this case. Then it was shown in \cite{PolRo96b} that the subgradient mapping for such functions is strictly proto-differentiable if and only if the proximal mapping is strictly differentiable. If the assumption of subdifferential continuity is omitted, the stated properties still hold for some $\varphi$-attentive $\epsilon$-localization of $\partial \varphi$.

Since the  introduction of strict proto-differentiability in \cite{PolRo96}, most contributions deal only with subgradient mappings for special function classes and its relation with strict twice epi-differentiability, see, e.g., \cite{HaJuSa22,HaSa23a,HaSa23} for some  recent papers. Very little investigations have been done for general multifunctions. In the very recent paper \cite{GfrOut24a} it is shown that for a graphically Lipschitzian and SCD-semismooth$^*$ multifunction, the set of all points belonging to its graph where strict proto-differentiability fails to hold, is negligible with respect to some Hausdorff measure.

Strict proto-differentiability has its roots in the notion of strictly smooth sets as introduced by Rockafellar \cite{Ro85}. We will show that a strictly smooth set coincides locally around the reference point, up to a permutation of the coordinates, with the graph of a strictly differentiable single-valued function, which implies in particular that the set is a Lipschitz manifold in the sense of \cite{Ro85}. This property carries over to strict proto-differentiability of multifunctions and allows its characterization by several generalized derivatives together with the graphical Lipschitz property. These characterizations can be used to show that under strict proto-differentiability and some condition on the dimension of the graphic Lipschitz property, the properties of strong metric regularity, metric regularity and strong metric subregularity are equivalent.

We will also consider the special case of the subgradient mapping of prox-regular functions. Besides various characterizations of strict proto-differentiability by means of $\varphi$-attentive generalized derivatives as  introduced in the very recent paper \cite{Gfr24b}, we also provide a new second-order relationship between function values and subgradients.

The paper is organized in the following way. In Section 2, which is divided into several subsections, we provide all the necessary definitions together with some easy consequences and well-known results. In Section 3 we state the characterizations of strictly smooth sets and strict proto-differentiability of mappings together with the equivalences between various regularity properties. The final Section 4 is devoted to the subgradient mapping of prox-regular functions.

Given a set-valued mapping $F:\R^n\tto\R^m$, we denote its domain by $\dom F:=\{x\in\R^n\mv F(x)\not=\emptyset\}$ and its graph by $\gph F:=\{(x,y)\in\R^n\times\R^m\mv y\in F(x)\}$. The space of real $m\times n$ matrices is denoted by $\R^{m\times n}$ and for a nonsingular matrix $A\in\R^{n\times n}$ we define $A^{-T}:=(A^{-1})^T$. For $A\in \R^{n\times d}$, $B\in \R^{m\times d}$ we denote by
\[\rge A:=\{Ap\mv p\in\R^d\},\quad \rge(A,B):=\{(Ap,Bp)\mv p\in\R^d\}\]
the subspaces of $\R^n$ and $\R^n\times\R^m$ generated by the ranges of $A$ and $(A,B)$, respectively. The dimension of a subspace $L$ of a finite-dimensional Hilbert space is denoted by $\dim L$ and $L^\perp$ signifies its orthogonal complement. Given a differentiable mapping $f:\R^n\to\R^m$, we write $\nabla f(x)$ for the Jacobian of $f$ at $x$. Hence, in case of a function ($m=1$) $\nabla f(x)$ is a row vector. The open ball with radius $r>0$ around a point $x\in\R^n$ is denoted by $\B_r(x)$. The polar cone to a set $A\subset\R^n$ is denoted by $A^\circ$. Given a parameterized family $C_t$ of subsets of a metric space, where $t$ also belongs to a metric space, we define the upper (outer) and lower (inner) limits in the sense of Painlev\'e--Kuratowski as
\begin{align*}
\limsup_{t\to\bar t}C_t:=\{x\mv \liminf_{t\to\bar t}\,\dist{x,C_t}=0\},\
\liminf_{t\to\bar t}C_t:=\{x\mv \limsup_{t\to\bar t}\,\dist{x,C_t}=0\}.
\end{align*}

\section{Basic definitions and preliminary results}

\subsection{Variational geometry and strictly smooth sets}

\begin{definition} Consider a set $\Omega\subset\R^n$ and a point $\zb\in\Omega$.
\begin{enumerate}
\item The {\em tangent cone}, the {\em regular (Clarke) tangent cone} and the {\em paratingent cone}  to $\Omega$ at $\zb$ are given by \vspace{-1em}
\begin{gather*}
T_\Omega(\zb):=\limsup_{t\downarrow 0}\frac{\Omega-\zb}t,\quad
\widehat T_\Omega(\zb):=\liminf_{\AT{z\setto{\Omega}\zb}{t\downarrow 0}}\frac{\Omega-z}t,\quad
T^P_\Omega(\zb):= \limsup_{\AT{z\setto{\Omega}\zb}{t\downarrow 0}}\frac{\Omega-z}t.
\end{gather*}

The set $\Omega$ is called {\em geometrically derivable} at $\zb$ if $T_\Omega =\lim_{t\downarrow 0}(\Omega-\zb)/t$
and $\Omega$ is said to be {\em strictly smooth} at $\zb$ if
\[\lim_{\AT{z\setto{\Omega}\zb}{t\downarrow 0}}\frac{\Omega-z}t\]
exists, i.e., when $\widehat T_\Omega(\zb)=T^P_\Omega(\zb)$.
\item  The {\em regular normal cone} and the {\em limiting normal cone} to $\Omega$ at $\zb$ are given by
\begin{gather*}
\widehat N_\Omega(\zb):=\big(T_\Omega(\zb)\big)^\circ,\quad
N_\Omega(\zb):=\limsup_{z\setto{\Omega}\zb}\widehat N_\Omega(z).
\end{gather*}
The set $\Omega$ is called {\em normally regular} at $\zb$ if $\widehat N_\Omega(\zb)=N_\Omega(\zb)$.
\end{enumerate}
\end{definition}
It is easy to see that the inclusions $\widehat T_\Omega(\zb)\subset T_\Omega(\zb)\subset T^P_\Omega(\zb)$
always hold.
Further, $\widehat T_\Omega(\zb)$ is a closed convex cone and $T^P_\Omega(\zb)$ is a closed cone having the property
$T^P_\Omega(\zb)=-T^P_\Omega(\zb)$.
Thus, $\Omega$ is strictly smooth at $\zb$ if and only if
\begin{equation}\label{EqStrictSmoothPrim}\widehat T_\Omega(\zb)=T^P_\Omega(\zb)(=T_\Omega(\zb))\end{equation}
is a subspace.

Clearly, $\widehat N_\Omega(\zb)\subset N_\Omega(\zb)$. Further, if $\Omega$ is locally closed at $\zb$ then the tangent-normal relation
\begin{equation}\label{EqTangNormal}
  \widehat T_\Omega(\zb)=N_\Omega(\zb)^\circ
\end{equation}
holds, cf. \cite[Theorem 6.28]{RoWe98}.
\begin{lemma}\label{LemStrictSmooth}
  Assume that the set $\Omega\subset\R^n$ is locally closed at $\zb$. Then $\Omega$ is strictly smooth at $\zb$ if and only if both $T^P_\Omega(\zb)$ and $N_\Omega(\zb)$ are subspaces satisfying
  \begin{equation}\label{EqPolarityStrictlySmooth}T^P_\Omega(\zb)^\perp=N_\Omega(\zb).\end{equation}
  Moreover, if $\Omega$ is strictly smooth at $\zb$  then $\Omega$ is normally regular at $\zb$.
\end{lemma}
\begin{proof} If $\Omega$ is strictly smooth,  then $T_\Omega^P(\zb)$ is a subspace and, by polarizing \eqref{EqStrictSmoothPrim} and taking into account  \eqref{EqTangNormal}, we obtain that
  \[T^P_\Omega(\zb)^\circ=T^P_\Omega(\zb)^\perp=T_\Omega(\zb)^\circ=\widehat N_\Omega(\zb)=\widehat T_\Omega(\zb)^\circ =\co N_\Omega(\zb)\]
  is a subspace. Since $\widehat N_\Omega(\zb)\subset N_\Omega(\zb)$, we conclude that $\co N_\Omega(\zb)=N_\Omega(\zb)=\widehat N_\Omega(\zb)=T^P_\Omega(\zb)^\perp$ is a subspace and $\Omega$ is normally regular at $\zb$. Conversely, if $T^P_\Omega(\zb)$ and $N_\Omega(\zb)$ are subspaces satisfying \eqref{EqPolarityStrictlySmooth}, by polarizing \eqref{EqPolarityStrictlySmooth} and using \eqref{EqTangNormal} we obtain that $T^P_\Omega(\zb)=N_\Omega(\zb)^\perp=N_\Omega(\zb)^\circ=\widehat T_\Omega(\zb)$ showing that $\Omega$ is strictly smooth at $\zb$.
\end{proof}

\subsection{Generalized differentiation and strict proto-differentiability of set-valued mappings}
\begin{definition}\label{DefStrictProto}
  Let $F:\R^n\tto\R^m$ be a mapping and let $(\xb,\yb)\in\gph F$.
  \begin{enumerate}
    \item The {\em strict derivative} $D_*F(\xb,\yb):\R^n\tto\R^m$, the {\em graphical derivative} $DF(\xb,\yb):\R^n\tto\R^m$ and the {\em limiting (Mordukhovich) coderivative} $D^*F(\xb,\yb):\R^n\tto\R^m$ at $\xb$ for $\yb$ are given by
        \begin{gather*}\gph D_*F(\xb,\yb)=T_{\gph F}^P(\xb,\yb),\quad \gph DF(\xb,\yb)=T_{\gph F}(\xb,\yb),\\
         \gph D^*F(\xb,\yb)=\{(y^*,x^*)\mv (x^*,-y^*)\in N_{\gph F}(\xb,\yb)\}.
         \end{gather*}
         \item The mapping $F$ is called {\em graphically regular} at $(\xb,\yb)$ if $\gph F$ is normally regular at $(\xb,\yb)$.
        \item The mapping $F$ is called {\em proto-differentiable} at $\xb$ for $\yb$ if $\gph F$ is geometrically derivable at $(\xb,\yb)$.
        \item The mapping $F$ is called {\em strictly proto-differentiable} at $\xb$ for $\yb$ if $\gph F$ is strictly smooth at $(\xb,\yb)$.
  \end{enumerate}
\end{definition}
In case when $F:\R^n\to\R^m$ is a single-valued mapping the notation can be simplified by omitting the second argument, e.g., we will write $DF(\xb)$ instead of $DF(\xb,F(\xb))$, etc.

In view of \eqref{EqStrictSmoothPrim}, $F:\R^n\tto\R^m$ is  strictly proto-differentiable at $\xb$ for $\yb\in F(\xb)$ if and only if
\begin{equation}
  \label{EqStrictProtPrim} \gph D_*F(\xb,\yb)=T^P_{\gph F}(\xb,\yb)=\widehat T_{\gph F}(\xb,\yb)(=T_{\gph F}(\xb,\yb)=\gph DF(\xb,\yb))
\end{equation}
is a subspace.

The definition of strict proto-differentiability goes back to \cite{PolRo96}, where it was introduced in a somewhat different way. In \cite{PolRo96}, a set-valued mapping $F:\R^n\tto\R^m$ was defined to be strictly proto-differentiable at $\xb$ for $\yb\in F(\xb)$, if the set-valued mappings
\[\triangle_{t,x,y}F:u\mapsto \frac{F(x+tu)-y}t\]
graphically converge if $t\downarrow 0$ and $(x,y)\longsetto{\gph F}(\xb,\yb)$. By substituting $x'=x+tu$ we see that
\[\gph \triangle_{t,x,y}F =\{\frac{(x'-x,y'-y)}t\mv (x',y')\in\gph F\}=\frac{\gph F-(x,y)}t\]
and hence this definition of strict proto-differentiability from \cite{PolRo96} coincides with the one of Definition \ref{DefStrictProto}. However, there is also another definition of strict proto-differentiability by Rockafellar \cite{Ro89}, which is equivalent to the condition $T_{\gph F}(\xb,\yb)=\widehat T_{\gph F}(\xb,\yb)$, cf. \cite[Proposition 2.7]{Ro89}. In view of \eqref{EqStrictProtPrim} we obtain that strict proto-differentiability in the sense of Definition \ref{DefStrictProto} implies this property in the sense of \cite{Ro89}, but the reverse implication does not hold in general. E.g., any mapping $F$ with closed convex graph having nonempty interior satisfies $T_{\gph F}(y,y)=\widehat T_{\gph F}(x,y)$, $(x,y)\in\gph F$, but $\widehat T_{\gph F}(x,y)$ is not a subspace for $(x,y)\in\bd\gph F$. A more subtle example is given by the mapping $F(x):=\{\pm x^2\}$, $x\in\R$ at $(0,0)$. Then easy calculations show that  $T_{\gph F}(0,0)=\widehat T_{\gph F}(0,0)$ equals to the subspace $\R\times\{0\}$, but $T^P_{\gph F}(0,0)=\R^2$.

 We now want to state the counterpart of Lemma \ref{LemStrictSmooth} for strictly proto-differentiable mappings.

Defining the $(n+m)\times(n+m)$ orthogonal matrix
\[S_{nm}:=\begin{pmatrix}0&-I_m\\I_n&0\end{pmatrix},\]
the graph of the limiting coderivative can also be expressed as
\[\gph D^*F(\xb,\yb)=S_{nm}N_{\gph F}(\xb,\yb).\]

\begin{lemma}\label{LemStrictProto}
  Let $F:\R^n\tto\R^m$ be a mapping having locally closed graph at $(\xb,\yb)\in\gph F$. Then $F$ is strictly proto-differentiable at $\xb$ for $\yb$ if and only if both $\gph D_*F(\xb,\yb)$ and $\gph D^*F(\xb,\yb)$ are subspaces satisfying
  \begin{equation*}
  \gph D^*F(\xb,\yb)=\{(y^*,x^*)\mv (x^*,-y^*)\in\gph D_*F(\xb,\yb)^\perp\}.
  \end{equation*}
  Moreover, in this case $F$ is graphically regular at $(\xb,\yb)$.
\end{lemma}
\begin{proof}
  By Lemma \ref{LemStrictSmooth}, $\gph F$ is strictly smooth at $(\xb,\yb)$ if and only if $T^P_{\gph F}(\xb,\yb)\big(=\gph D_*F(\xb,\yb)\big)$ and $N_{\gph F}(\xb,\yb)$ are subspaces satisfying
  $N_{\gph F}(\xb,\yb)=\gph D_*F(\xb,\yb)^\perp$, which in turn is equivalent to the condition that $\gph D_*F(\xb,\yb)$ and $\gph D^*F(\xb,\yb)=S_{nm}N_{\gph F}(\xb,\yb)$ are subspaces and
  \[\gph D^*F(\xb,\yb)=S_{nm}N_{\gph F}(\xb,\yb)=S_{nm}T^P_{\gph F}(\xb,\yb)^\perp=\{(y^*,x^*)\mv (x^*,-y^*)\in\gph D_*F(\xb,\yb)^\perp\}.\]
  From Lemma \ref{LemStrictSmooth} we may also conclude that $F$ is graphically regular at $(\xb,\yb)$.
\end{proof}

So far we have only considered generalized derivatives well-known  from the literature. We now call  attention to other generalized derivatives, the so-called SC (subspace containing) derivatives, which were very recently introduced in \cite{GfrOut22}. Though the development of SC derivatives was mainly motivated by an efficient implementation of the semismooth$*$ Newton method for solving generalized equations from \cite{GfrOut21}, they also were useful for analyzing various theoretical topics in variational analysis, cf. \cite{GfrOut23, GfrOut24a, Gfr24b}.

Let $\tilde \Z_{nm}$  denote the metric space of all subspaces of $\R^n\times \R^m$ equipped with the metric
\[d_\Z(L_1,L_2)=\norm{P_{L_1}-P_{L_2}},\]
where $P_{L_i}$, $i=1,2$, denotes the orthogonal projection on $L_i$. Further, we denote by $\Z_{nm}$ the metric space of all subspaces of $\R^n\times\R^m$ of dimension $n$ equipped with the metric $d_\Z$. By \cite[Exercise 5.35]{RoWe98}, a sequence $L_k\in \tilde \Z_{nm}$ converges to some subspace $L\in \tilde \Z_{nm}$ with respect to $d_\Z$, i.e.., $\lim_{k\to\infty}d_\Z(L_k,L)=0$, if and only if $L_k$ converges to $L$ in the sense of set convergence. In this case we must have $\dim L =\dim L_k$ for all $k$ sufficiently large and hence $\Z_{nm}$ is closed in $\tilde\Z_{nm}$. Since orthogonal projections are uniformly bounded, both $\tilde \Z_{nm}$ and $\Z_{nm}$ are compact.

Given a subspace $L\in \tilde \Z_{nm}$, we denote by
\[L^*:=S_{nm}L^\perp =\{(y^*,x^*)\in\R^m\times\R^n\mv (x^*,-y^*)\in L^\perp\}\in\tilde \Z_{mn}.\]
its {\em adjoint }subspace. Then $(L^*)^*=L$ and $d_\Z(L_1,L_2)=d_\Z(L_1^*,L_2^*)$, cf. \cite{GfrOut22}. Since $\dim L^*=\dim L^\perp = n+m-\dim L$, we have $L^*\in\Z_{mn}$ whenever $L\in\Z_{nm}$.

\begin{definition}\label{DefSCD}Let $F:\R^n\tto\R^m$ be a mapping.
  \begin{enumerate}
    \item $F$ is called {\em graphically smooth} at $(x,y)\in\gph F$ and of dimension $d$ in this respect, if $T_{\gph F}(x,y)$ is a $d$-dimensional subspace of $\R^n\times\R^m$. Further we denote by $\tilde \OO_F$ the set of all points from the graph of $F$, where $F$ is graphically smooth. The set $\OO_F$ is defined as the subset of $\tilde\OO_F$ where $F$ is graphically smooth of dimension $n$.
    \item The  SC derivative $\Sp F:\gph F\tto  \Z_{nm}$ and the generalized SC derivative $\tilde \Sp F:\gph F\tto \tilde \Z_{nm}$ are defined by
    \begin{gather*}
      \Sp F(x,y):=\{L\in \Z_{nm}\mv \exists (x_k,y_k)\longsetto {\OO_F}(x,y): d_\Z(T_{\gph F}(x_k,y_k),L)=0\},\\
      \tilde \Sp F(x,y):=\{L\in\tilde \Z_{nm}\mv \exists (x_k,y_k)\longsetto {\tilde \OO_F}(x,y): d_\Z(T_{\gph F}(x_k,y_k),L)=0\},
    \end{gather*}
    whereas the adjoint SC derivatives   $\Sp^*F:\gph F\tto \Z_{mn}$ and $\tilde \Sp^*F:\gph F\tto\tilde \Z_{mn}$ are given by
    \[\Sp^*F(x,y):=\{L^*\mv L\in\Sp F(x,y)\},\quad\tilde \Sp^*F(x,y):=\{L^*\mv L\in\tilde \Sp F(x,y)\}.\]
    \item We say that $F$ has the {\em  SCD (subspace containing derivative) property} and {\em generalized SCD property at} $(\xb,\yb)\in\gph F$  if $\Sp F(\xb,\yb)\not=\emptyset$ and $\tilde \Sp F(\xb,\yb)\not=\emptyset$, respectively.\\
        We say that $F$ has the {\em (generalized) SCD property} {\em around}  $(\xb,\yb)\in\gph F$, if $F$ has this property at all points $(x,y)\in\gph F$ sufficiently close to $(\xb,\yb)$.\\
        Finally we say that $F$ is a (generalized) SCD mapping if $\Sp F(x,y)\not=\emptyset$, $(x,y)\in\gph F$ ($\tilde \Sp F(x,y)\not=\emptyset$, $(x,y)\in\gph F$).
  \end{enumerate}
\end{definition}

\begin{remark}
  The SC derivatives $\Sp F, \Sp^*F$ were introduced in \cite{GfrOut22, GfrOut23}, whereas the definitions of their generalized counterparts are new. Many of the properties of $\Sp F$ and
  $\Sp^*F$ established in \cite{GfrOut22, GfrOut23} carry easily over to $\tilde \Sp F$ and $\tilde \Sp^*F$ and we will use these properties without proving them again. Let us mention that for many mappings $F$ important in practice there holds $\tilde \Sp F =\Sp F$. We will discuss this feature more in detail below.
\end{remark}

Obviously the inclusion $\Sp F(x,y)\subset\tilde \Sp F(x,y)$ holds. One also has  the following relation between the SC derivative and the strict derivative and limiting coderivative, cf. \cite{GfrOut22}: For every $L\in\tilde \Sp F(x,y)$ there holds
\begin{equation}\label{EqInclSCD}L\subset \gph D_*F(x,y),\quad L^*\subset \gph D^*F(x,y).\end{equation}
 The generalized (adjoint) derivative may be therefore considered as a kind of skeleton for the strict derivative (limiting coderivative).

 \subsection{On strict differentiability of single-valued mappings}
 First we  collect some basic facts about single-valued mappings.

 \begin{lemma}\label{LemDimSubspSingleValued}
 Consider a mapping $F:U\to\R^m$, where $U\subset\R^n$ is open, and let $\xb\in U$.
 \begin{enumerate}
   \item[(i)] If $F$ is strictly continuous at $\xb$ and $\gph D_*F(\xb)$ is a subspace, then $\dim \gph D_*F(\xb)=n$ and $D_*F(\xb)$ is a linear mapping from $\R^n$ to $\R^m$.
   \item[(ii)] If $F$ is strictly continuous at $\xb$ and $\gph D^*F(\xb)$ is a subspace, then $\dim \gph D^*F(\xb)=m$ and $D^*F(\xb)$ is a linear mapping from $\R^m$ to $\R^n$.
   \item[(iii)] If $F$ is calm at $\xb$ and $\gph DF(\xb)$ is a subspace, then $\dim \gph DF(\xb)=n$ and $DF(\xb)$ is a linear mapping from $\R^n$ to $\R^m$.
 \end{enumerate}
 \end{lemma}
 \begin{proof}
   ad (i): By strict continuity of $F$ there is some neighborhood $U'$ of $\xb$ along with some real $\kappa\geq 0$ such the $\norm{F(x)-F(x')}\leq \kappa \norm{x-x'}$, $x,x'\in U'$. By the definition of the strict derivative, there holds $(u,v)\in \gph D_*F(\xb)$ if and only if there are sequences $x_k\to\xb$, $u_k\to u$ and $t_k\downarrow 0$ such that
   $(F(x_k+t_ku_k)-F(x_k))/t_k\to v$. Since  $\norm{F(x_k+t_ku_k)-F(x_k)}/t_k\leq \kappa\norm{u_k}$, we conclude that $\norm{v}\leq \kappa\norm{u}$. Hence, $\gph D_*F(\xb)$ cannot contain elements of the form $(0,v)\in\R^n\times\R^m$ with $v\not=0$. Thus $\dim \gph D_*F(\xb)\leq n$ and for every $u\in\R^n$ there can be at most one $v\in\R^m$ with $(u,v)\in \gph D_*F(\xb)$. On the other hand, for every $u\in\R^n$ and every sequence $t_k\downarrow 0$ the sequence $(F(\xb+t_ku)-F(\xb))/t_k$ is bounded and, by possibly passing to a subsequence, we have $(F(\xb+t_ku)-F(\xb))/t_k\to v\in D_*F(\xb)(u)$. Hence $\dom D_*F(\xb)=\R^n$ implying $\dim \gph D_*F(\xb)\geq n$. We conclude that $\dim \gph D_*F(\xb)= n$ and that $D_*F(\xb)$ is a single-valued mapping from $\R^n$ to $\R^m$, which must be necessarily linear.

   ad (ii): By strict continuity of $F$ at $\xb$ together with the Mordukhovich criterion \cite[Theorem 9.40]{RoWe98} we obtain that $\gph D^*F(\xb)$ cannot contain elements $(0,u^*)\in\R^m\times\R^n$ with $u^*\not=0$. Hence, $\dim \gph D^*F(\xb)\leq m$ and for every $v^*\in\R^m$ the set $D^*F(\xb)(v^*)$ can contain at most one element. On the other hand, by \cite[Proposition 9.24(b)]{RoWe98} the coderivative $D^*F(\xb)$ is nonempty-valued and consequently $\dim \gph D^*F(\xb)=m$ and $D^*F(\xb)$ is a single-valued mapping which is linear.

   ad (iii): This can be shown analogously to (i).
 \end{proof}
 By combining Lemma \ref{LemDimSubspSingleValued}(iii) with \cite[Exercise 9.25(b)]{RoWe98} we obtain the following corollary.
 \begin{corollary}\label{CorFrechetDiff}
   A mapping $F:U\to\R^m$, where $U\subset \R^n$ is open, is Fr\'echet differentiable at $\xb\in U$, if and only if $F$ is calm at $\xb$ and $\gph DF(\xb)=T_{\gph F}(\xb,F(\xb))$ is a subspace of $\R^n\times\R^m$. Moreover, in this case there holds $\dim T_{\gph F}(\xb,F(\xb))=n$.
 \end{corollary}
 \if{
 \begin{lemma}\label{LemFrechetDiff}
   A mapping $F:U\to\R^m$, where $U\subset \R^n$ is open, is Fr\'echet differentiable at $\xb\in U$, if and only if $F$ is calm at $\xb$ and $T_{\gph F}(\xb,F(\xb))$ is a subspace of $\R^n\times\R^m$. Moreover, in this case there holds $\dim T_{\gph F}(\xb,F(\xb))=n$
 \end{lemma}
 \begin{proof}
   The ''only if''-part follows immediately from \cite[Exercise 9.25]{RoWe98} and in this case we have $T_{\gph F}(\xb,F(\xb))=\rge(I_n,\nabla F(\xb))$. Now assume that $F$ is calm at $\xb$ and that $T_{\gph F}(\xb,F(\xb))$ is a subspace. From the assumed calmness we may deduce that there is some real $\kappa\geq0$ such that $\norm{v}\leq\kappa \norm{u}$, $\forall (u,v)\in T_{\gph F}=\gph DF(\xb)$ by \cite[Proposition 9.24]{RoWe98}. Hence $\dim T_{\gph F}(\xb,F(\xb))\leq n$ because otherwise there would exist $0\not=v\in \R^m$ with $(0,v)\in T_{\gph F}(\xb,F(\xb))$. On the other hand, calmness of $F$ also implies that for every $u\in\R^n$ we can find some sequence $t_k\downarrow 0$ such that $F((\xb+t_ku)-F(\xb))/t_k$ converges to some $v$. Then $(u,v)\in T_{\gph F}(\xb,F(\xb))$ and therefore $\dim T_{\gph F}(\xb,F(\xb))\geq n$ and $\dom DF(\xb)=\R^n$. Hence, $\dim T_{\gph F}(\xb,F(\xb))=n$ and $DF(\xb)$ is a linear mapping, verifying that $F$ is Fr\'echet differentiable at $\xb$.
 \end{proof}
}\fi
 Next consider a Lipschitzian mapping $F:U\to\R^m$, where $U\subset\R^n$ is open. Let $D_F$ denote the set of all points $x\in U$ where $F$ is differentiable. Then $T_{\gph F}(x)=\rge(I_n,\nabla F(x))\in \Z_{nm}$, $x\in D_F$ and from Corollary \ref{CorFrechetDiff} we deduce that $\tilde\OO_F=\OO_F=\{(x,F(x))\mv x\in D_F\}$.
 By Rademacher's theorem, the set $U\setminus D_F$ has Lebesgue measure $0$. Recall the definition of the {\em B-Jacobian} (Bouligand) $\onabla F(x)$ which reads as
 \[\onabla F(x):=\{A\in\R^{m\times n}\mv \exists x_k\longsetto{D_F}x: A=\lim_{k\to\infty}\nabla F(x_k)\}.\]
 Since the Jacobians $\nabla F(x)$, $x\in D_F$ are bounded by the Lipschitz constant of $F$ and $D_F$ is dense in $U$, there holds $\onabla F(x)\not=\emptyset$, $x\in U$.
 For a convergent sequence $\nabla F(x_k)\to A$ we have that $\rge(I_n,\nabla F(x_k))\to\rge(I_n,A)$ and it follows that for every $x\in U$ there holds
 \begin{equation}\label{EqSCD_StrictDiff}\Sp F(x)=\tilde \Sp F(x)=\{\rge(I_n,A)\mv A\in\onabla F(x)\},\quad \Sp^* F(x)=\tilde \Sp^* F(x)=\{\rge(I_m,A^T)\mv A\in\onabla F(x)\},\end{equation}
 cf. \cite[Lemma 3.5]{GfrOut23}.

Strict differentiability will play an important role in this paper. Recall that a mapping $F:U\to\R^m$, $U\subset\R^n$ open, is strictly differentiable at $\xb\in U$ if there is an $m\times n$ matrix $A$ such that
\[\lim_{\AT{x,x'\to\xb}{x\not=x'}}\frac{F(x')-F(x)-A(x'-x)}{\norm{x'-x}}=0.\]
This implies that $F$ is differentiable at $\xb$ and $\nabla F(\xb)=A$. Further, by the definition of the strict derivative and \cite[Exercise 9.25(c)]{RoWe98}, respectively, strict differentiability can be identified  with the linearity of either $D_*F(\xb)$ or  $D^*F(\xb)$ and in this case one has
\begin{equation}\label{EqStricTDiffCalc}D_*F(\xb)(u)=\nabla F(\xb)u,\ u\in\R^n,\quad D^*F(\xb)(y^*)=\nabla F(\xb)^Ty^*,\ y^*\in\R^m.\end{equation}

In the following lemma we summarize some further characterizations of strict differentiability.
\begin{lemma}\label{LemStrictDiff}
  Consider a single-valued mapping $F:U\to\R^m$, where $U\subset\R^n$ is open, and a point $\xb\in U$. Then the following statements are equivalent:
  \begin{enumerate}
    \item[(i)] $F$ is strictly differentiable at $\xb$.
    \item[(ii)] $F$ is strictly continuous at $\xb$ and $F$ is strictly proto-differentiable at $\xb$.
    \item[(iii)] $F$ is both strictly continuous and graphically regular at $\xb$.
    \item[(iv)]  $F$ is strictly continuous at $\xb$ and $\onabla F(\xb)$ is a singleton.
    \item[(v)] $F$ is strictly continuous at $\xb$  and $\tilde\Sp F(\xb)$ is a singleton.
    \item[(vi)] $F$ is strictly continuous at $\xb$  and $\tilde\Sp^* F(\xb)$ is a singleton.
    \item[(vii)]  $F$ is strictly continuous at $\xb$ and $\gph D_*F(\xb)$ is a subspace.
    \item[(viii)] $F$ is strictly continuous at $\xb$ and $\gph D^*F(\xb)$ is a subspace.
  \end{enumerate}
  In every case there holds $\dim \gph D_*F(\xb)=n$, $\dim \gph D^*F(\xb)=m$, $\tilde \Sp F(\xb)=\Sp F(\xb)=\{\gph D_*F(\xb)\}$ and $\tilde \Sp^* F(\xb)=\Sp^* F(\xb)=\{\gph D^*F(\xb)\}$.
\end{lemma}
\begin{proof}
  ad (i)$\Rightarrow$(ii),(vii),(viii): Clearly, if $F$ is strictly differentiable at $\xb$, it is strictly continuous at $\xb$. Strict proto-differentiability and the subspace property of $\gph D_*F(\xb)$, $\gph D^*F(\xb)$ follow now from \eqref{EqStricTDiffCalc} and  Lemma \ref{LemStrictProto}.

  ad (ii)$\Rightarrow$(iii): Follows from Lemma \ref{LemStrictProto}.

  ad (iii)$\Leftrightarrow$(i): Follows from \cite[Exercise 9.25(d)]{RoWe98}.

  ad (i)$\Leftrightarrow$(iv): See, e.g., \cite[Lemma 2.1]{GfrOut24a}.

  ad (iv)$\Leftrightarrow$(v)$\Leftrightarrow$(vi): Follows from \eqref{EqSCD_StrictDiff}.

  ad (vii)$\Rightarrow$(i): If (vii) holds then $D_*F(\xb)$ is a linear mapping by Lemma \ref{LemDimSubspSingleValued}(i) and strict differentiability follows from the definitions.

  ad (viii)$\Rightarrow$(i): If (viii) holds then $D^*F(\xb)$ is a linear mapping by Lemma \ref{LemDimSubspSingleValued}(ii) and strict differentiability follows from  \cite[Exercise 9.25(c)]{RoWe98}.

  The assertions about the dimensions of $\gph D_*F(\xb)$ and $\gph D^*F(\xb)$ follow from Lemma \ref{LemDimSubspSingleValued}, the assertions about the SC derivatives follow from \eqref{EqInclSCD} and \eqref{EqSCD_StrictDiff}.
\end{proof}

\subsection{On graphically Lipschitzian and semismooth$^*$ mappings}
The following definitions will play an important role.
\begin{definition}\label{DefLipMan}
\begin{enumerate}
 \item[(i)]  A set $\Omega\subset \R^n$ is called a {\em Lipschitz manifold of dimension $d$ around $\zb\in\Omega$}  if there is an open
neighborhood $W$ of $\zb$ and a one-to-one mapping $\Phi$ between $W$ and  an open subset of $\R^n$ with $\Phi$ and $\Phi^{-1}$  continuously differentiable, such that $\Phi(\Omega\cap W)$ is the graph of a Lipschitz continuous mapping $f:U\to\R^{n-d}$,  where $U$ is an open set in $\R^d$.

\item[(ii)] A mapping $F:\R^n\tto\R^m$ is called {\em graphically Lipschitzian of dimension $d$ around $(\xb,\yb)\in\gph F$} if $\gph F$ is a Lipschitz manifold of dimension $d$ around $(\xb,\yb)$.
\item[(iii)] A mapping $F:\R^n\tto\R^m$ is called {\em graphically strictly differentiable} at $(\xb,\yb)\in\gph F$, and  of dimension $d$ in this respect, if $\gph F$ is a Lipschitz manifold of dimension $d$ around $(\xb,\yb)$ and, using the notation from (i), the mapping $f:U\to\R^{n+m-d}$ is strictly differentiable at $\bar u$, where $(\bar u,f(\bar u))=\Phi(\bar x,\bar y)$.
    \end{enumerate}
\end{definition}
Thus, a mapping $F$ is graphically Lipschitzian (graphically strictly differentiable) at some point belonging to its graph, if, up to some change of coordinates, the graph of $F$ can be locally identified with the graph of a Lipschitzian (strictly differentiable) single-valued mapping.
\begin{remark}
\begin{enumerate}
\item   The definition  above of a Lipschitz manifold is due to Rockafellar \cite{Ro85}. However, the reader should be aware that there are also other definitions of Lipschitz manifolds in the literature, see, e.g., \cite{Pen22}.
\item Our notion of graphical strict differentiability of a mapping $F:\R^n\tto\R^m$ is called graphical smoothness in \cite{Mo06a}. Thus, graphical smoothness in the sense of \cite{Mo06a} is a much stronger property than graphical smoothness in the sense of our Definition \ref{DefSCD}. Further, Rockafellar \cite{Ro85} defined a set $\Omega$ to be smooth at $\zb\in\Omega$ if $\Omega$ is geometrically derivable at $\zb$ and $T_\Omega(\zb)$ is a subspace. Hence, if $\gph F$ is smooth at $(\xb,\yb)\in\gph F$ in the sense of Rockafellar \cite{Ro85}, $F$ is graphically smooth at $(\xb,\yb)$ in the sense of our Definition \ref{DefSCD}, but the reverse implication will not hold in general.
\end{enumerate}
\end{remark}
An important class of graphically Lipschitzian mappings is formed by locally maximal hypo-monotone multifunctions. This class includes, in particular, the subdifferential mapping of so-called {\em prox-regular} functions typically encountered in finite-dimensional optimization. We refer the reader to the book of Rockafellar and Wets \cite{RoWe98} for more details and discussions. If a mapping $F:\R^n\tto\R^m$ is graphically Lipschitzian of dimension $n$ at $(\xb,\yb)\in\gph F$, then it has the SCD-property around $(\xb,\yb)$, cf. \cite[Proposition 3.17]{GfrOut22}. Further, one can show that $\Sp F(x,y)=\tilde \Sp F(x,y)$ for all $(x,y)\in\gph F$ close to $(\xb,\yb)$.

Consider a set $\Omega\subset\R^n$ which is a Lipschitz manifold of dimension $d$ around a point $\zb\in\Omega$ and let $\Phi$ and $f$ be as in Definition \ref{DefLipMan}. Then
\begin{gather}\label{EqTangConeChangeCoord}T_{\gph f}(\Phi(\zb))=\nabla \Phi(\zb)T_\Omega(\zb),\ \widehat T_{\gph f}(\Phi(\zb))=\nabla \Phi(\zb)\widehat T_\Omega(\zb),\ T^P_{\gph f}(\Phi(\zb))=\nabla \Phi(\zb)T^P_\Omega(\zb),\\
\label{EqNormalConeChangeCoord}\widehat N_{\gph f}(\Phi(\zb))=\nabla \Phi(\zb)^{-T}\widehat N_\Omega(\zb),\  N_{\gph f}(\Phi(\zb))=\nabla \Phi(\zb)^{-T} N_\Omega(\zb).
\end{gather}
\begin{lemma}\label{LemDimLipMan}
     Assume that $\Omega\subset\R^n$ is a Lipschitz manifold of dimension $d$ around $\zb\in\Omega$. If $T^P_\Omega(\zb)$ is a subspace then $\dim T^P_\Omega(\zb)=d$. Further, if $N_\Omega(\zb)$ is a subspace then its dimension is $n-d$.
\end{lemma}
\begin{proof}
  Let $\Phi$ and $f$ be as in Definition \ref{DefLipMan} and let $\bar u$ be given by $(\bar u,f(\bar u))=\Phi(\zb)$. If $T^P_\Omega(\zb)$ is a subspace, then $\gph D_*f(\bar u)=T^P_{\gph f}(\Phi(\zb))$ is a subspace with the same dimension by \eqref{EqTangConeChangeCoord} and together with  Lemma \ref{LemDimSubspSingleValued}(i) we conclude that $\dim T^P_\Omega(\zb)=d$. If $N_\Omega(\zb)$ is a subspace, then $N_{\gph f}(\Phi(\zb))$ and $\gph D^*f(\ub)$ are subspaces with the same dimension by \eqref{EqNormalConeChangeCoord} and from Lemma \ref{LemDimSubspSingleValued}(ii) the equality $\dim N_\Omega(\zb)=n-d$ follows.
\end{proof}

If $F:\R^n\tto\R^m$ is graphically Lipschitzian of dimension $d$ at $(\xb,\yb)\in\gph F$, then with the corresponding mappings $\Phi$ and $f$ there hold the relations
\begin{gather}\label{EqTransGraphDer}\gph Df(\bar u)=\nabla \Phi(\xb,\yb)\gph DF(\xb,\yb),\ \gph D_*f(\bar u)=\nabla \Phi(\xb,\yb)\gph D_*F(\xb,\yb),\\
\nonumber\gph D^*f(\bar u)=S_{d,n+m-d}\nabla \Phi(\xb,\yb)^{-T}S_{nm}^T\gph D^*F(\xb,\yb),\\
\label{EqTransSCDDer}\tilde \Sp f(\bar u)=\Sp f(\bar u)=\nabla \Phi(\xb,\yb)\tilde \Sp F(\xb,\yb),\ \tilde \Sp^* f(\bar u)=\Sp^* f(\bar u)=S_{d,n+m-d}\nabla \Phi(\xb,\yb)^{-T}S_{nm}^T\tilde \Sp^*F(\bar u),
\end{gather}
where $\bar u$ is given by $(\bar u,f(\bar u))=\Phi(\xb,\yb)$.

Let us now recall the definitions of \ssstar sets and mappings.

\begin{definition}
\begin{enumerate}
\item  A set $\Omega\subseteq\R^n$ is called {\em \ssstar} at a point $\zb\in \Omega$ if for every $\epsilon>0$ there is some $\delta>0$ such that
  \begin{equation*}
\vert \skalp{z^*,z-\zb}\vert\leq \epsilon \norm{z-\zb}\norm{z^*}\ \forall z\in \Omega\cap\B_\delta(\zb)\
\forall z^*\in N_\Omega(z).
\end{equation*}
\item
A set-valued mapping $F:\R^n\tto\R^m$ is called {\em \ssstar} at a point $(\xb,\yb)\in\gph F$, if
$\gph F$ is \ssstar at $(\xb,\yb)$, i.e.,   for every $\epsilon>0$ there is some $\delta>0$ such that
\begin{align*}
\nonumber&\lefteqn{\vert \skalp{x^*,x-\xb}-\skalp{y^*,y-\yb}\vert}\\
&\qquad\leq \epsilon
\norm{(x,y)-(\xb,\yb)}\norm{(y^*,x^*)}\ \forall(x,y)\in \gph F\cap\B_\delta(\xb,\yb)\ \forall
(y^*,x^*)\in\gph D^*F(x,y).
\end{align*}
\end{enumerate}
\end{definition}
Note that the semismooth$^*$ property was defined in \cite{GfrOut21} in a different but equivalent way. The class of semismooth$^*$ mappings is rather broad.
We mention here two important classes of multifunctions having this property.
\begin{proposition}[cf. {\cite[Proposition 2.10]{GfrOut22}}]
  \begin{enumerate}
    \item[(i)]Every mapping whose graph is the union of finitely many closed convex sets is \ssstar at every point of its graph.
    \item[(ii)]Every mapping with closed subanalytic graph is \ssstar at every point of its graph.
  \end{enumerate}
\end{proposition}

For single-valued functions there holds the following result.

\begin{proposition}[cf. {\cite[Proposition 3.7]{GfrOut21}}]\label{PropNewtondiff}
Assume that $F:\R^n\to \R^m$ is a single-valued mapping
which is Lipschitzian near $\xb$. Then the following two statements are equivalent.
\begin{enumerate}
  \item[(i)] $F$ is \ssstar at $\xb$.
  \item[(ii)] For every $\epsilon>0$ there is some $\delta>0$ such that
  \begin{equation*}
    \norm{F(x)-F(\xb)-C(x-\xb)}\leq\epsilon\norm{x-\xb}\ \forall x\in \B_\delta(\xb)\ \forall
    C\in\co \overline\nabla F(x).
  \end{equation*}
\if{  i.e.,
  \begin{equation}\label{EqCharSemiSmoothLipsch1}
    \sup_{C\in \co\overline\nabla F(x)}\norm{F(x)-F(\xb)-C(x-\xb)}=\oo(\norm{x-\xb}).
  \end{equation}
}\fi
\end{enumerate}
\end{proposition}
By \cite[Corollary 4.3]{GfrOut24a}, graphically Lipschitzian and \ssstar mappings are almost everywhere strictly proto-differentiable:
\begin{theorem}
  Assume that the mapping $F:\R^n\tto\R^m$ is graphically Lipschitzian of dimension $n$ at $(\xb,\yb)\in\gph F$ and is \ssstar at all $(x,y)\in\gph F$ close to $(\xb,\yb)$. Then there is an open neighborhood $W$ of $(\xb,\yb)$ such that for almost all $(x,y)\in\gph F\cap W$ (with respect to the $n$-dimensional Hausdorff measure) the mapping $F$ is strictly proto-differentiable at $(x,y)$.
\end{theorem}

\subsection{On regularity of set-valued mappings}

The last subsection is devoted to various regularity properties.

\begin{definition}
  Let $F:\R^n\tto\R^m$ be a mapping and let $(\xb,\yb)\in\gph F$.
  \begin{enumerate}
  \item $F$ is said to be {\em strongly metrically subregular at} $(\xb,\yb)$ if there exists $\kappa\geq 0$  along with some  neighborhood $U$ of $\xb$ such that
    \begin{equation*}
      \norm{x-\xb}\leq \kappa\dist{\yb,F(x)}\ \forall x\in U.
    \end{equation*}
  \item $F$ is said to be {\em metrically regular around} $(\xb,\yb)$ if there is $\kappa\geq 0$ together with neighborhoods $U$ of $\xb$ and $V$ of $\yb$ such that
      \begin{equation*}
      \dist{x,F^{-1}(y)}\leq \kappa\dist{y,F(x)}\ \forall (x,y)\in U\times V.
    \end{equation*}
  \item $F$ is said to be {\em strongly metrically regular around} $(\xb,\yb)$ if it is metrically regular around $(\xb,\yb)$ and $F^{-1}$ has a single-valued localization around $(\yb,\yb)$, i.e., there are  open neighborhoods $V'$ of $\yb$, $U'$ of $\xb$ and a mapping $h:V'\to\R^n$ with $h(\yb)=\xb$ such that $\gph F\cap (U'\times V')=\{(h(y),y)\mv y\in V'\}$. By metric regularity of $F$, $h$ is actually strictly continuous at $\yb$.
  \end{enumerate}
\end{definition}

In this paper we will also use the  following point-based characterizations of the above regularity properties.
\begin{theorem}\label{ThCharRegByDer}
    Let $F:\R^n\tto\R^m$ be a mapping whose graph is locally closed at  $(\xb,\yb)\in\gph F$.
    \begin{enumerate}
    \item[(i)] (Levy-Rockafellar criterion, see, e.g., \cite[Theorem 4E.1]{DoRo14}) $F$ is strongly metrically subregular at $(\xb,\yb)$ if and only if
      \begin{equation}\label{EqLevRoCrit}
        0\in DF(\xb,\yb)(u)\ \Rightarrow u=0.
      \end{equation}
    \item[(ii)] (Mordukhovich criterion, see, e.g., \cite[Theorem 3.3]{Mo18}) $F$ is metrically regular around $(\xb,\yb)$ if and only if
      \begin{equation}
            \label{EqMoCrit} 0\in D^*F(\xb,\yb)(y^*)\ \Rightarrow\ y^*=0.
      \end{equation}
    \item[(iii)](\cite[Theorem 2.7]{GfrOut22}) $F$ is strongly metrically regular around $(\xb,\yb)$ if and only if
      \begin{equation}
            \label{EqStrictCrit} 0\in D_*F(\xb,\yb)(u)\ \Rightarrow\ u=0
      \end{equation}
      and \eqref{EqMoCrit} holds.
    \end{enumerate}
\end{theorem}

\section{Point-based Characterizations of strict proto-differentiability and consequences}
We start with the following proposition.
\begin{proposition}\label{PropLipMan}
  Assume that $\Omega\subset\R^n$ is locally closed and strictly smooth at $\zb\in \Omega$. Then $\Omega$ is a Lipschitz manifold of dimension $d$ around $\zb$, where $d$ is the dimension of the subspace $T_\Omega(\zb)$. Actually, there is an $n\times n$ permutation matrix $\tilde Q$ and an open neighborhood $V\times U\subset\R^d\times\R^{n-d}$ of $\tilde Q\zb=:(\yb,\xb)\in\R^d\times\R^{n-d}$ such that $\tilde Q\Omega\cap(V\times U)$ is the graph of a Lipschitzian mapping $f:V\to\R^{n-d}$ which is strictly differentiable at $\yb$.
\end{proposition}
\begin{proof}
  Consider an $n\times d$ matrix $Z$ whose columns form a basis for $T_\Omega(\zb)$. Then the row rank of $Z$ is also $d$  and we can find an $n\times n$ permutation matrix $Q$ such that the last $d$ rows of $QZ$ are linearly independent. Thus, if we partition the $n\times d$ matrix $QZ$ into  matrices $A\in\R^{(n-d)\times d}$ and $B\in\R^{d\times d}$, $QZ=\left(\begin{smallmatrix}A\\B\end{smallmatrix}\right)$, the matrix $B$ is nonsingular. Consider the mapping $G:\R^{n-d}\tto\R^d$ given by $\gph G=Q\Omega$ and we claim that $G$ is strongly metrically regular around $Q\zb=:(\xb,\yb)\in\R^{n-d}\times\R^d$. Indeed, since the $n\times n$ matrix $Q$ is nonsingular, we have
  \[T_{\gph G}(\xb,\yb)=QT_\Omega(\zb)=QT^P_\Omega(\zb)=T^P_{\gph G}(\xb,\yb)=\gph D_*G(\xb,\yb)=\rge (A,B)=\rge(AB^{-1},I_d)\]
 and consequently, by taking into account the identity $\big(QT^P_\Omega(\zb)\big)^\perp=Q^{-T}T^P_\Omega(\zb)^\perp=QT^P_\Omega(\zb)^\perp=QN_\Omega(\zb)$,
 \begin{gather*}N_{\gph G}(\xb,\yb)= Q^{-T}N_\Omega(\zb)=QN_\Omega(\zb)=\big(QT^P_\Omega(\zb)\big)^\perp=\rge \big(AB^{-1},I_d\big)^\perp=\rge\big(I_{n-d},-B^{-T}A^T\big),\\
  \gph D^*G(\xb,\yb)=S_{n-d,d}N_{\gph G}(\xb,\yb)=\rge\big(B^{-T}A^T, I_{n-d}\big).
 \end{gather*}
 Thus, for every element $(u,0)\in\gph D_*G(\xb)(\xb,\yb)$ we have $u=0$ and for every $(v^*,0)\in\gph D^*G(\xb,\yb)$ we have $v^*=0$ and we conclude from Theorem \ref{ThCharRegByDer} that $G$ is strongly metrically regular around $(\xb,\yb)$. Thus there exists an open neighborhood $U\times V$ of $(\xb,\yb)$ and a Lipschitzian mapping $f:V\to\R^{n-d}$ such that $\gph f=\gph G^{-1}\cap (V\times U)$. By defining  the permutation matrix
 \[\tilde Q:=\begin{pmatrix}0&I_d\\I_{n-d}&0\end{pmatrix}Q\]
 and the  orthogonal transformation $\Phi(z)=\tilde Q z$,
 we obtain that
 \[\gph f=\gph G^{-1}\cap(V\times U)= \Phi(\Omega)\cap(V\times U)=\Phi\big(\Omega\cap\Phi^{-1}(V\times U)\big)\]
 and, since $\Phi^{-1}(V\times U)$ is an open neighborhood of $\zb$, it  follows that $\Omega$ is a $d$-dimensional Lipschitz manifold. Since $\Omega$ is strictly smooth, $T^P_\Omega(\xb)$ and consequently $T^P_{\gph f}(\yb,\xb)=\gph D_*f(\yb)$ are subspaces and the assertion about the strict differentiability of $f$ at $\yb$ follows from Lemma \ref{LemStrictDiff}.
\end{proof}

We are now in the position to state point-based characterizations of strict smoothness of sets and strict proto-differentiability of mappings, respectively.

\begin{theorem}\label{ThStrictlySmooth}
  Assume that the set $\Omega\subset\R^n$ is locally closed at $\zb\in\Omega$. Then the following statements are equivalent:
  \begin{enumerate}
    \item[(i)] $\Omega$ is strictly smooth at $\zb$.
    \item[(ii)] There is some nonnegative integer $d$ such that $\Omega$ is a Lipschitz manifold of dimension $d$ around $\zb$ and $\widehat T_\Omega(\zb)$ is an $d$-dimensional subspace.
    \item[(iii)]  $\Omega$ is a Lipschitz manifold around $\zb$  and $T_\Omega^P(\zb)$ is a subspace.
    \item[(iv)]  $\Omega$ is a Lipschitz manifold around $\zb$ and $N_\Omega(\zb)$ is a subspace.
    \item[(v)] $\Omega$ is a Lipschitz manifold around $\zb$ and normally regular at $\zb$.
  \end{enumerate}
  In every case there holds $\dim T^P_\Omega(\zb)=\dim T_\Omega(\zb)= \dim \widehat T_\Omega(\zb)=d$ and $\dim N_\Omega(\zb)=n-d$, where $d$ is the dimension of the Lipschitz manifold $\Omega$.
\end{theorem}
\begin{proof}
  The implications  (i)$\Rightarrow$(iii), (i)$\Rightarrow$(iv) and (i)$\Rightarrow$(v) follow from Proposition \ref{PropLipMan}, \eqref{EqStrictSmoothPrim} and Lemma \ref{LemStrictSmooth}. From Lemma \ref{LemDimLipMan} and \eqref{EqTangNormal} we can deduce that (iv)$\Rightarrow$(ii). The implication (ii)$\Rightarrow$(i) follows from a result by Rockafellar \cite[Theorem 3.5(c)]{Ro85}. Thus it remains to show that (iii)$\Rightarrow$(i) and (v)$\Rightarrow$(i). Assume that either (iii) or (v) holds. In both cases,  $\Omega$ is a Lipschitzian manifold around $\zb$. Let  $\Phi,\ U$ and $f$ be as in Definition \ref{DefLipMan} and let $\xb\in U$ be given by $\Phi(\zb)=(\xb,f(\xb))$. If (iii) holds then $\gph D_*f(\xb)=T^P_{\gph f}(\Phi(\zb))=\nabla\Phi(\zb)T_\Omega^P(\zb)$
  is a  subspace of $\R^n$ and hence $f$ is strictly proto-differentiable at $\xb$ by Lemma \ref{LemStrictDiff}. On the other hand, if (v) holds then $N_{\gph f}(\Phi(\zb))=\nabla \Phi(\zb)^TN_\Omega(\zb)=\nabla \Phi(\zb)^T\widehat N_\Omega(\zb)=\widehat N_{\gph f}(\Phi(\zb))$ implying that $f$ is graphically regular at $\xb$ and consequently  also strictly proto-differentiable by Lemma \ref{LemStrictDiff}. Thus, $\gph f$ is strictly smooth at $\Phi(\zb)$ and $T^P_{\gph f}(\Phi(\zb))=\widehat T_{\gph f}(\Phi(\zb))$ is a subspace. Hence $T^P_\Omega(\zb)=\nabla \Phi(\zb)^{-1}T^P_{\gph f}(\Phi(\zb))= \nabla \Phi(\zb)^{-1}\widehat T_{\gph f}(\Phi(\zb))=\widehat T_\Omega(\zb)$ is a subspace and it follows that $\Omega$ is strictly smooth at $\zb$.

  The assertions about the dimensions of the tangent cones and the limiting normal cone follow from \eqref{EqStrictSmoothPrim} and Lemma \ref{LemDimLipMan}.
  \end{proof}

\begin{corollary}\label{CorStrictlyProt}
  Assume that the mapping $F:\R^n\tto\R^m$ has locally closed graph at $(\xb,\yb)\in\gph F$. Then the following statements are equivalent:
  \begin{enumerate}
    \item[(i)] $F$ is strictly proto-differentiable at $\xb$ for $\yb$.
    \item[(ii)] $F$ is graphically strictly differentiable at $(\xb,\yb)$.
    \item[(iii)] $F$ is graphically Lipschitzian of dimension $d$ around $(\xb,\yb)$ and $\widehat T_{\gph F}(\xb,\yb)$ is a $d$-dimensional subspace.
    \item[(iv)] $F$ is graphically Lipschitzian  around $(\xb,\yb)$ and $\gph D_*F(\xb,\yb)$ is a  subspace.
    \item[(v)] $F$ is graphically Lipschitzian around $(\xb,\yb)$ and $\gph D^*F(\xb,\yb)$ is a subspace.
    \item[(vi)] $F$ is both graphically Lipschitzian around $(\xb,\yb)$ and graphically regular at $(\xb,\yb)$.
    \item[(vii)] $F$ is graphically Lipschitzian around $(\xb,\yb)$ and $\tilde \Sp F(\xb,\yb)$ is a singleton.
    \item[(viii)] $F$ is graphically Lipschitzian around $(\xb,\yb)$ and $\tilde \Sp^* F(\xb,\yb)$ is a singleton.
  \end{enumerate}
  In every case there holds $\dim \gph D_* F(\xb,\yb)=\dim \gph DF(\xb,\yb)=\dim \widehat T_{\gph F}(\xb,\yb)=d$, $\dim \gph D^*F(\xb,\yb)=n+m-d$,  where $d$ is the dimension of the graphical Lipschitz property of $F$, and $\tilde \Sp F(\xb,\yb)=\{\gph DF(\xb,\yb)\}$,  $\tilde \Sp^* F(\xb,\yb)=\{\gph D^*F(\xb,\yb)\}$.
\end{corollary}
\begin{proof}In order to prove the equivalences between the assertions (ii)--(viii) just note that in every case the mapping $F$ is graphically Lipschitzian around $(\xb,\yb)$. Let $\Phi, U$ and $f$ be as in Definition \ref{DefLipMan} and let $\bar u\in U$ be given by $\Phi(\xb,\yb)=(\bar u,f(\bar u))$. Then, by \eqref{EqTransGraphDer}-\eqref{EqTransSCDDer} and Lemma \ref{LemStrictDiff}, all the statements (iii)-(viii) are equivalent to strict differentiability of $f$ at $\bar u$ which in turn is (ii) by definition. The equivalences between (i) and any of the statements (iii)-(vi) follow from Theorem \ref{ThStrictlySmooth}.
\end{proof}
\begin{remark}
  \begin{enumerate}
    \item We have only stated point-based characterizations of strict smoothness and strict proto-differentiability, but the list of equivalences in Theorem \ref{ThStrictlySmooth} and Corollary \ref{CorStrictlyProt} can be extended. E.g., by combining Proposition \ref{PropLipMan} with \cite[Theorem 2.3]{HaSa23} one immediately obtains that a set $\Omega\subset\R^n$, locally closed at $\zb\in \Omega$, is strictly smooth at $\zb$ if and only if $\Omega$ is a Lipschitz manifold around $\zb$ and $\lim_{z\setto\Omega\zb}T_\Omega(z)= T_\Omega(\zb)$.
    \item The equivalence (i)$\Leftrightarrow$(ii) in Theorem \ref{ThStrictlySmooth} sharpens  Rockafellar's characterization \cite[Theorem 3.5]{Ro85} of strictly smooth sets insofar as the Lipschitz manifold property of the set has not to be imposed as an assumption.
  \end{enumerate}
\end{remark}
Strictly smooth sets and strictly proto-differentiable mappings are always \ssstar at the reference point.
\begin{theorem}\begin{enumerate}\item[(i)] Assume that the set $\Omega\subset\R^n$ is locally closed and strictly smooth at $\zb\in\Omega$. Then $\Omega$ is also \ssstar at $\zb$.
\item[(ii)] Assume that the mapping $F:\R^n\tto\R^m$ has locally closed graph at $(\xb,\yb)\in\gph F$ and is strictly proto-differentiable at $\xb$ for $\yb$. Then $F$ is \ssstar at $(\xb,\yb)$.
\end{enumerate}
\end{theorem}
\begin{proof}
  We prove assertion (i) by contraposition. Assume that $\Omega$ is not \ssstar at $\zb$ and there exists some $\epsilon>0$ together with sequences $z_k\setto{\Omega}\zb$ and $z_k^*\in N_\Omega(z_k)$ such that
  $\vert \skalp{z_k^*,z_k-\zb}\vert>\epsilon\norm{z_k^*}\norm{z_k-\zb}$. By passing to a subsequence if necessary we may assume that the sequence $(z_k-\zb)/\norm{z_k-\zb}$ converge to some $w\in T_\Omega(\zb)\subset T_\Omega^P(\zb)$ and the sequence $z_k^*/\norm{z_k^*}$ converge to some $z^*\in N_\Omega(\zb)$. Then $\vert\skalp{z^*,w}\vert>\epsilon$ and consequently $\Omega$ is not strictly smooth at $\zb$ by Lemma \ref{LemStrictSmooth}. This proves (i) and the second assertion follows immediately from the definitions.
\end{proof}
Under strict proto-differentiability, various regularity properties are equivalent.

\begin{theorem}\label{ThStrictProtoReg}
  Assume that $F:\R^n\tto\R^m$ is strictly proto-differentiable at $\xb$ for $\yb\in F(\xb)$ and assume that $\gph F$ is locally closed at $(\xb,\yb)$. Then the following statements are equivalent:
  \begin{enumerate}
  \item[(i)] $F$ is strongly metrically regular around $(\xb,\yb)$.
  \item[(ii)] $F$ is graphically Lipschitzian of dimension $m$  and metrically regular around $(\xb,\yb)$.
  \item[(iii)] $F$ is graphically Lipschitzian of dimension $m$  around $(\xb,\yb)$ and strongly metrically subregular at $(\xb,\yb)$.
  \item[(iv)] There is an $n\times m$ matrix $C$ such that
  \[\gph D^*F(\xb,\yb)=\rge(C^T, I_n).\]
  \item[(v)] There is an $n\times m$ matrix $C$ such that
  \[\gph DF(\xb,\yb)=\rge(C,I_m).\]
  \end{enumerate}
  In every case $F^{-1}$ has a single-valued localization around $(\yb,\xb)$ which is strictly differentiable at $\yb$.
\end{theorem}
\begin{proof}

ad (i)$\Rightarrow$(ii),(iii): If (i) holds then $F^{-1}$ has a single-valued Lipschitzian localization around $(\yb,\xb)$. Thus, the property that $F$ is graphically Lipschitzian of dimension $m$ can be certified with the bijection $\Phi(x,y)=(y,x)$, $(x,y)\in\R^n\times\R^m$ and strict differentiability of the localization follows from Corollary \ref{CorStrictlyProt}. Obviously, strong metric regularity implies both metric regularity and strong metric subregularity.

ad (ii)$\Leftrightarrow$(iv): Assume that (ii) holds. Since $F$ is graphically Lipschitzian of dimension $m$ around $(\xb,\yb)$ and is assumed to be strictly proto-differentiable at $\xb$ for $\yb$, we conclude from Corollary \ref{CorStrictlyProt}  that $\gph D^*F(\xb,\yb)$ is an $n$ dimensional subspace of $\R^m\times\R^n$.  Choose an $m\times n$ matrix $A$ and an $n\times n$ matrix $B$ such that $\gph D^*F(\xb,\yb)=\rge(A,B)$ and we claim that $B$ is nonsingular. Assume on the contrary that there is some $0\not=p\in\R^n$ such that $Bp=0$. Then $0=Bp\in D^*F(\xb,\yb)(Ap)$ and by the Mordukhovich criterion \eqref{EqMoCrit} we conclude that $Ap=0$. Hence $(Ap,Bp)=(0,0)$, $p\not=0$ and therefore $\dim \rge(A,B)<n$, a contradiction.
Hence, $B$ is nonsingular and  $\gph D^*F(\xb,\yb)=\rge(A,B)=\rge(AB^{-1},I_n)$.\\
Conversely, if (iv) holds it is easy to see that \eqref{EqMoCrit} is fulfilled and hence $F$ is metrically regular around $(\xb,\yb)$. Further, $\dim \gph D^*F(\xb,\yb)=n$ and we conclude from Corollary \ref{CorStrictlyProt} that $F$ is graphically Lipschitzian of dimension $m$ around $(\xb,\yb)$.

ad (iii)$\Leftrightarrow$(v): This can be shown similarly to the equivalence (ii)$\Leftrightarrow$(iv) by replacing the Mordukhovich criterion \eqref{EqMoCrit} with the Levy-Rockafellar criterion \eqref{EqLevRoCrit} criterion.

ad (iv)$\Leftrightarrow$(v): Since $F$ is strictly proto-differentiable, by Lemma \ref{LemStrictProto} both $\gph D^*F(\xb,\yb)$ and $\gph D_*F(\xb,\yb)$ are subspaces and
$\gph D^*F(\xb,\yb)=(\gph D_*F(\xb,\yb))^*$.  Further we have $\gph D_*F(\xb,\yb)=\gph DF(\xb,\yb)$ by \eqref{EqStrictProtPrim}. Thus, if (v) holds, then $\gph D_*F(\xb,\yb)^\perp=\gph DF(\xb,\yb)^\perp=\rge(I_n,-C^T)$ and $\gph D^*F(\xb,\yb)=S_{nm}\gph D_*F(\xb,\yb)^\perp=\rge(C^T,I_n)$. Conversely, if (iv) holds, then $\gph D_*F(\xb,\yb)=(\gph D^*F(\xb,\yb))^*=\gph(C,I_m)$.

ad (iv)$\Rightarrow$(i): As we have just proven, statement (iv) is equivalent with (v). Statement (iv) implies that the Mordukhovich criterion \eqref{EqMoCrit}  is fulfilled and (v) together with the equation  $\gph D_*F(\xb,\yb)=\gph DF(\xb,\yb)$ imply that the condition \eqref{EqStrictCrit} hold. Hence, $F$ is strongly metrically regular around $(\xb,\yb)$ by Theorem \ref{ThCharRegByDer}(iii).
\end{proof}

At the end of this section we want to specialize this result for mappings of the form
\begin{equation}\label{EqGE}F(x)=g(x)+G(x),\end{equation}
where $g:\R^n\to\R^m$ is single-valued and $G:\R^n\tto\R^m$ is set-valued. If $g$ is strictly differentiable at $\xb$, we obtain from \cite[Exercise 10.43]{RoWe98} for any $\yb\in F(\xb)$ the representations
\begin{gather}\label{EqGEStrictDer}\gph D_*F(\xb,\yb)=\begin{pmatrix}I_n&0\\\nabla g(\xb)& I_m\end{pmatrix}\gph D_*G(\xb,\yb-g(\xb)),\\
 \label{EqGECoderiv}\gph D^*F(\xb,\yb)=\begin{pmatrix}I_m&0\\\nabla g(\xb)^T& I_n\end{pmatrix}\gph D^*G(\xb,\yb-g(\xb)).\end{gather}
Further, it follows from the definition that $F$ is strictly proto-differentiable at $\xb$ for $\yb$ if and only if $G$ is strictly proto-differentiable at $\xb$ for $\yb-g(\xb)$.

\begin{corollary}Consider the mapping $F$ given by \eqref{EqGE} and let $(\xb,\yb)\in\gph F$. Assume that $g$ is strictly differentiable at $\xb$ and assume that $G$ has closed graph at $(\xb,\yb-g(\xb))$ and is strictly proto-differentiable at $\yb-g(\xb)$. Then the following properties are equivalent.
  \begin{enumerate}
  \item[(i)] $F$ is strongly metrically regular around $(\xb,\yb)$.
  \item[(ii)] $G$ is graphically Lipschitzian of dimension $m$ around $(\xb,\yb)$ and $F$ is metrically regular around $(\xb,\yb)$.
  \item[(iii)] $G$ is graphically Lipschitzian of dimension $m$  around $(\xb,\yb)$ and $F$ is strongly metrically subregular at $(\xb,\yb)$.
  \item[(iv)] There are matrices $A\in\R^{n\times m},\ B\in\R^{m\times m}$ such that $\gph D_*G(\xb,\yb-g(\xb))=\rge(A,B)$ and the $m\times m$ matrix $\nabla g(x)A+B$ is nonsingular.
  \item[(v)] There are matrices $\tilde A\in\R^{m\times n},\ \tilde B\in\R^{n\times n}$ such that $\gph D^*G(\xb,\yb-g(\xb))=\rge(\tilde A,\tilde B)$ and the $n\times n$ matrix $\nabla g(x)^T\tilde A+\tilde B$ is nonsingular.
  \end{enumerate}
 \end{corollary}
 \begin{proof}
    Note that by \eqref{EqGEStrictDer}, \eqref{EqGECoderiv} we have $\dim\gph D_*F(\xb,\yb)=\dim\gph D_*G(\xb,\yb-g(\xb))$, $\dim\gph D^*F(\xb,\yb)=\dim\gph D^*G(\xb,\yb-g(\xb))$. Further, in the cases (iv) and (v) there hold
    \begin{gather*}
      \gph D_*F(\xb,\yb)=\rge(A,\nabla g(x)A+B)=\rge\big(A(\nabla g(x)A+B)^{-1},I_m\big),\\
      \gph D^*F(\xb,\yb)=\rge(\tilde A,\nabla g(x)^T\tilde A+\tilde B)=\rge\big(\tilde A(\nabla g(x)^T\tilde A+\tilde B)^{-1},I_n\big).
    \end{gather*}
    The assertion follows now from Theorem \ref{ThStrictProtoReg}.
 \end{proof}
 \begin{remark}
   In \cite{HaSa23} the authors proved equivalence of the different regularity properties for the special case when $G$ is the subgradient mapping of a prox-regular and subdifferentially continuous function $\varphi:\R^n\to(-\infty,\infty]$ under the additional assumption that $\nabla g(\xb)=\nabla g(\xb)^T$. The assumptions on $\varphi$ guarantee that $\partial\varphi$ is graphically Lipschitzian of dimension $n$, cf. \cite[Theorem 4.7]{PolRo96} and therefore the requirement $\nabla g(\xb)=\nabla g(\xb)^T$ is superfluous.
 \end{remark}

\section{Strict proto-differentiability of the subgradient mapping}

Given an extended real-valued function $\varphi:\R^n\to\oR:=(-\infty,\infty]$ and  a point $\xb\in \dom  \varphi:=\{x\in\R^n\mv  \varphi(x)<\infty\}$, the {\em regular subdifferential} of $\varphi$ at $\xb$ is given by
\[\widehat\partial  \varphi(\xb):=\Big\{x^*\in\R^n\mv\liminf_{x\to\xb}\frac{ \varphi(x)- \varphi(\xb)-\skalp{x^*,x-\xb}}{\norm{x-\xb}}\geq 0\Big\},\]
while the {\em limiting (Mordukhovich) subdifferential} is defined by
\[\partial  \varphi(\xb):=\{x^*\mv \exists x_k\attto{ \varphi} \xb, x_k^* \to x^* \mbox{ with }x_k^*\in\widehat \partial  \varphi(x_k)\ \forall k\},\]
where $x_k\attto{ \varphi} \xb$ denotes {\em $\varphi$-attentive} convergence, i.e., $x_k\attto{ \varphi} \xb \ :\Longleftrightarrow\ x_k\to \xb \mbox{ and }  \varphi(x_k)\to  \varphi(\xb).$

 An lsc function $\varphi:\R^n\to\oR$ is called {\em prox-regular} at $\xb$ for $\xba$
 if $\varphi$ is finite  at $\xb$ with $\xba\in\partial \varphi(\xb)$ and  there
exist $\epsilon> 0$ and $r\geq 0$ such that
\[ \varphi(x')\geq  \varphi(x)+\skalp{x^*,x'-x}-\frac r2 \norm{x'-x}^2\]
 whenever $x'\in \B_\epsilon(\xb)$ and $(x,x^*)\in \gph\partial \varphi\cap (\B_\epsilon(\xb)\times \B_\epsilon(\xba))$ with $\varphi(x)<\varphi(\xb)+\epsilon$.
 The function $\varphi$ is called {\em subdifferentially continuous} at $\xb$ for $\xba\in\partial f(\xb)$ if for every sequence $(x_k,x_k^*)\longsetto{\gph\partial\phi}(\xb,\xba)$ we have $\varphi(x_k)\to\varphi(\xb)$. In  absence of subdifferential continuity we cannot expect that $\gph\partial \varphi$ is locally closed and we should work with a $\varphi$-attentive $\epsilon$-localization around $(\xb,\xba)$, which is  the multifunction $\Ta:\R^m\tto\R^n$ defined by
 \[\gph\Ta=\{(x,x^*)\in\gph \partial \varphi\cap\B_\epsilon(\xb)\times\B_\epsilon(\xba)\mv \varphi(x)<\varphi(\xb)+\epsilon\}.\]
 It is well-known \cite[Theorem 4.7]{PolRo96}, that $\Ta$ is graphically Lipschitzian of dimension $n$ around $(\xb,\xba)$, whenever $\varphi$ is prox-regular at $\xb$ for $\xba$ with parameters $r\geq 0$ and $\epsilon>0$. As transformation mapping $\Phi$ one can take $\Phi(x,x^*):=(x+\lambda x^*,x)$ with $\lambda\in(0,\frac 1r)$ and the corresponding Lipschitz mapping $f$ is given by the {\em proximal mapping}
 \begin{equation*}
 f(u):=P_\lambda(u):=(I+\lambda \Ta)^{-1}(u).
 \end{equation*}
 As an alternative transformation mapping one can take $\tilde \Phi(x,x^*)=(x+\lambda x^*,x^*)$ yielding the Lipschitzian mapping
 \begin{equation}\label{EqAltGradMorEnv}
   \tilde f(u):=(u-P_\lambda(u))/\lambda.
 \end{equation}
 Moreover $\gph \Ta$ is locally closed at $(\xb,\xba)$, cf. \cite[Lemma 3.4]{Gfr24b}. Having these properties in mind, straightforward application of Corollary \ref{CorStrictlyProt} yields characterizations of strict proto-differentiability of $\Ta$ by means of generalized derivatives of $\Ta$. However, the mapping $\Ta$ depends on the parameter $\epsilon$ of prox-regularity and we often do not know it a priori. But there exists the remedy of introducing $\varphi$-attentive generalized derivatives which differ from their ordinary counterparts by considering $\varphi$-attentive convergence defined as
 \begin{equation*}
(x_k,x_k^*)\attconv{\varphi}(\xb,\xba)\ :\Longleftrightarrow\ (x_k,x_k^*)\longsetto{\gph \partial \varphi}(\xb,\xba) \mbox{ and }\varphi(x_k)\to \varphi(\xb).
\end{equation*}

 \begin{definition}[cf. \cite{Gfr24b}]
  Let $\varphi:\R^n\to\oR$ be a function and $(\xb,\xba)\in \gph \partial \varphi$.
  \begin{enumerate}
    \item The {\em $\varphi$-attentive tangent cone} to $\gph \partial \varphi$ at $(\xb,\xba)$ is given by
    \begin{equation*}
      T^\varphi_{\gph \partial \varphi}(\xb,\xba):=\left\{(u,u^*)\mv \exists t_k\downarrow 0,\ (x_k,x_k^*)\attconv{\varphi}(\xb,\xba): (u,u^*)=\lim_{k\to\infty}\frac{(x_k,x_k^*)-(\xb,\xba)}{t_k}\right\}.
    \end{equation*}
    \item The {\em $\varphi$-attentive regular normal cone} to $\gph \partial \varphi$ at $(\xb,\xba)$ is defined as $\widehat N^\varphi_{\gph\partial \varphi}(\xb,\xba):=\big( T^\varphi_{\gph\partial \varphi}(\xb,\xba)\big)^\circ$ and the {\em $\varphi$-attentive limiting normal cone} to $\gph \partial \varphi$ at $(\xb,\xba)$ amounts to
        \[N^\varphi_{\gph\partial \varphi}(\xb,\xba):=\Limsup_{(x,x^*)\attconv{\varphi}(\xb,\xba)} \widehat N^\varphi_{\gph\partial \varphi}(x,x^*).\]
    \item The {\em $\varphi$-attentive graphical derivative $D_\varphi(\partial \varphi)(\xb,\xba)$} and the {\em $\varphi$-attentive limiting  coderivative $D_\varphi^*(\partial \varphi)(\xb,\xba)$} of $\gph \partial \varphi$ at $(\xb,\xba)$ are the set-valued mappings from $\R^n$ to the subsets of $\R^n$ given by
        \begin{gather*}
          \gph D_\varphi(\partial \varphi)(\xb,\xba):=T^\varphi_{\gph \partial \varphi}(\xb,\xba),\
          \gph  D_\varphi^*(\partial \varphi)(\xb,\xba) :=\{(u,u^*)\mv (u^*,-u)\in N^\varphi_{\gph\partial \varphi}(x,x^*)\}.
        \end{gather*}
    \item Let $\OO_{\partial \varphi}^\varphi$ denote the collection of all points $(x,x^*)\in\gph \partial \varphi$ such that $T^\varphi_{\gph \partial \varphi}(x,x^*)$ is an $n$-dimensional subspace. Then the {\em $\varphi$-attentive SC derivative} of $\partial \varphi$ at $(\xb,\xba)$ is defined by
        \[\Sp_\varphi(\partial \varphi)(\xb,\xba):=\{L\in\Z_{nn}\mv \exists (x_k,x_k^*)\subset \OO^\varphi_{\partial \varphi}, (x_k,x_k^*)\attconv{\varphi}(\xb,\xba): d_{\Z}\big(L,T^\varphi_{\gph \partial \varphi}(x_k,x_k^*)\big)\to0\}.\]
        Finally the {\em adjoint $\varphi$-attentive SC derivative} of $\partial \varphi$ at $(\xb,\xba)$ is defined by
        \[\Sp_\varphi^*(\partial \varphi)(\xb,\xba):=\{L^*\mv L\in \Sp_\varphi(\partial \varphi)(\xb,\xba)\}.\]
  \end{enumerate}
\end{definition}
If the lsc function $\varphi:\R^n\tto\oR$ is prox-regular at $\xb$ for $\xba\in\partial\varphi(\xb)$ with parameter $\epsilon>0$ and $\Ta$ denotes the $\varphi$-attentive $\epsilon$-localization of $\partial\varphi$ around $(\xb,\xba)$ then for every $(x,x^*)\in\gph \Ta$ there holds
\begin{gather*}T^\varphi_{\gph \partial \varphi}(x,x^*)=T_{\gph \Ta}(x,x^*),\ \widehat N^\varphi_{\gph \partial \varphi}(x,x^*)=\widehat N_{\gph \Ta}(x,x^*),\ N^\varphi_{\gph \partial \varphi}(x,x^*)=N_{\gph \Ta}(x,x^*),\\
  D_\varphi(\partial \varphi)(x,x^*)=D\Ta(x,x^*),\ D_\varphi^*(\partial \varphi)(x,x^*)=D^*\Ta(x,x^*),\\
  \emptyset\not=\Sp_\varphi(\partial \varphi)(x,x^*)=\Sp\Ta(x,x^*)=\Sp_\varphi^*(\partial \varphi)(x,x^*)=\Sp^*\Ta(x,x^*),
\end{gather*}
cf. \cite[Proposition 3.5]{Gfr24b}.
\if{
Moreover. for every $L\in \Sp_\varphi(\partial \varphi)(x,x^*)$ there holds $L=L^*$ and there are unique symmetric $n\times n$ matrices $P,W$ such that
\begin{equation}
  L=\rge(P,W),\ P^2=P,\ W(I-P)=I-P.
\end{equation}
}\fi
Putting everything together we obtain from Corollary \ref{CorStrictlyProt} the following statement.
  \begin{corollary}\label{CorStrictlyProtSubdiff}
  Assume that the lsc function $\varphi:\R^n\tto\R^n$ is prox-regular at $\xb$ for $\xba\in\partial \varphi(\xb)$ with parameter $\epsilon>0$ and let $\Ta$ denote the $\varphi$-attentive $\epsilon$-localization of $\partial\varphi$ around $(\xb,\xba)$. Then the following statements are equivalent:
  \begin{enumerate}
    \item[(i)] $\Ta$  is strictly proto-differentiable at $\xb$ for $\xba$.
    \item[(ii)] $\Ta$ is graphically strictly differentiable at $(\xb,\xba)$.
    \item[(iii)]  $\gph D_\varphi^*(\partial\varphi)(\xb,\xba)$ is a subspace.
    \item[(iv)] $\Ta$ is graphically regular at $(\xb,\xba)$, i.e., $\widehat N_{\gph\partial\varphi}^\varphi(\xb,\xba)= N_{\gph\partial\varphi}^\varphi(\xb,\xba)$.
    \item[(v)] $\Sp_\varphi (\partial\varphi)(\xb,\xba)$ is a singleton.
  \end{enumerate}
  In every case there holds  $\Sp_\varphi \partial \varphi(\xb,\xba)=\Sp_\varphi^* \partial \varphi(\xb,\xba)= \{\gph D_\varphi(\partial\varphi)(\xb,\xba)\}=\{\gph D^*_\varphi(\partial \varphi)(\xb,\xba)\}$.\\
  If, in addition, $\varphi$ is subdifferentially continuous at $\xb$ for $\xba$, then $\Ta$ can be replaced by $\partial\varphi$.
\end{corollary}
Strict proto-differentiability of $\Ta$ for prox-regular functions $\varphi$ is equivalent with the so-called property of strict twice epi-differentiability of $\varphi$, cf. \cite[Theorem 6.1]{PolRo96}. We will now state another second-order relation between  $\varphi$ and $\Ta$.
\begin{theorem}Assume that the lsc function $\varphi:\R^n\tto\R^n$ is prox-regular at $\xb$ for $\xba\in\partial \varphi(\xb)$ with parameters $r\geq 0$ and $\epsilon>0$ and let $\Ta$ denote the $\varphi$-attentive $\epsilon$-localization of $\partial\varphi$ around $(\xb,\xba)$.
\begin{enumerate}
  \item[(i)] If  $\Ta$ is either proto-differentiable at $\xb$ for $\xba$ or \ssstar at $(\xb,\xba)$ then
  \begin{equation}\label{EqTrapRule1}
    \lim_{(x,x^*)\attconv{\varphi}(\xb,\xba)}\frac{\varphi(x)-\varphi(\xb)-\frac 12\skalp{x^*+\xba,x-\xb}}{\norm{(x,x^*)-(\xb,\xba)}^2}=0.
  \end{equation}
   \item[(ii)] If  $\Ta$ is strictly proto-differentiable at $\xb$ for $\xba$ then
  \begin{equation}\label{EqTrapRule2}
    \lim_{\AT{(x,x^*),(y,y^*)\attconv{\varphi}(\xb,\xba)}{(x,x^*)\not=(y,y^*)}}\frac{\varphi(y)-\varphi(x)-\frac 12\skalp{y^*+x^*,y-x}}{\norm{(y,y^*)-(x,x^*)}^2}=0.
  \end{equation}
\end{enumerate}
\end{theorem}
\begin{proof}
  Since we are interested only in local properties of $\varphi$ and $\partial\varphi$, we may add to $\varphi$ the indicator function of some compact neighborhood of $\xb$ so that
  \[\varphi(x)\geq \varphi(\xb)+\skalp{\xba, x-x^*}-\frac r2\norm{x-\xb}^2,\ x\in\R^n.\]
  Consider the {\em Moreau envelope}
  \[e_\lambda(u):=\min_x\{\varphi(x)+\frac1{2\lambda}\norm{x-u}^2\}\]
  for fixed $\lambda\in (0,\frac 1r)$. It is well known \cite[Theorem 4.4]{PolRo96} that there exists an open convex neighborhood $U$ of $\bar u=\xb+\lambda \xba$ such that the proximal mapping $P_\lambda$ is single-valued and Lipschitz on $U$ and $e_\lambda$ is continuously differentiable on $U$ with Lipschitzian gradient
  $\nabla e_\lambda(u)=(u-P_\lambda(u))/\lambda$.

  ad (i): We will show in a first step that
  \begin{equation}\label{EqTrapRuleMorEnv1}\lim_{u\to\bar u}\frac{e_\lambda(u)-e_\lambda(\bar u)-\frac 12(\nabla e_\lambda(u)+\nabla e_\lambda(\bar u))(u-\bar u)}{\norm{u-\bar u}^2}=0.
 \end{equation}
Consider for arbitrary $u\in U$ the function $\psi_u:[0,1]\to\R$ given by $\psi_u(t)=\nabla e_\lambda(\ub+t(u-\ub))(u-\ub)$. Then
  \begin{equation}\label{EqFuncDiffInt}e_\lambda(u)-e_\lambda(\ub)-\nabla e_\lambda(\ub)(u-\ub)=\int_0^1\psi_u(t)-\psi_u(0){\rm d}t.\end{equation}

  Assume first that $\Ta$ is proto-differentiable at $\xb$ for $\xba$. It follows from \cite[Exrecise 13.45]{RoWe98} that $\nabla e_\lambda$ is semidifferentiable at $\bar u$, which means that the directional derivatives
  \[d(h):=\lim_{h'\to h,\ t\downarrow0}\frac{\nabla e_\lambda(\bar u+th')-\nabla e_\lambda(\bar u)}t,\
   h\in \R^n\]
  exist. Consider arbitrary $\epsilon>0$ and choose $\delta>0$ such that $\B_\delta(\bar u)\subset U$ and for every $h\in\R^n$ with $\norm{h}=1$ there holds
  \begin{equation}\label{EqAuxProto1}\Norm{\tau d(h)- (\nabla e_\lambda(\bar u+\tau h)-\nabla e_\lambda(\bar u))}\leq\tau\epsilon, \tau\in [0,\delta].
  \end{equation}
  For arbitrarily fixed $u\in \B_\delta(\ub)$, $u\not=\ub$ and $t\in [0,1]$ we have
  \begin{align*}
  \norm{\psi_u(t)-\psi_u(0)&-t(\psi_u(1)-\psi_u(0))}\leq  \Norm{\psi_u(t)-\psi_u(0)-t\norm{u-\ub}d\left(\frac{u-\ub}{\norm{u-\ub}}\right)(u-\ub)} \\&+t\Norm{\norm{u-\ub}d\left(\frac{u-\ub}{\norm{u-\ub}}\right)(u-\ub)-(\psi_u(1)-\psi_u(0))}
  \leq 2t\epsilon\norm{u-\ub}^2,
  \end{align*}
  where we have used the Cauchy-Schwarz inequality and \eqref{EqAuxProto1} with $h=(u-\ub)/\norm{u-\ub}$, $\tau=t\norm{u-\ub}$ and $\tau=\norm{u-\ub}$. Thus, by virtue of \eqref{EqFuncDiffInt} we obtain that
  \begin{align*}&\vert e_\lambda(u)-e_\lambda(\ub)-\nabla e_\lambda(\ub)(u-\ub)-\frac{\nabla e_\lambda(u)-\nabla e_\lambda(\ub)}2(u-\ub)\vert=\vert \int_0^1\psi_u(t)-\psi_u(0){\rm d}t -\frac {\psi_u(1)-\psi_u(0)}2\vert\\
  &=\vert \int_0^1\psi_u(t)-\psi_u(0)-t(\psi_u(1)-\psi_u(0)){\rm d}t\vert \leq \int_0^12t\epsilon\norm{u-\ub}^2{\rm d}t=\epsilon\norm{u-\ub}^2\end{align*}
  and we conclude that \eqref{EqTrapRuleMorEnv1} holds.

  Now assume that $\Ta$ is \ssstar at $(\xb,\xba)$. By taking into account \eqref{EqAltGradMorEnv}, it follows from \cite[Proposition 4.2]{GfrOut24a} that $\nabla e_\lambda$ is \ssstar at $\bar u$. By Proposition \ref{PropNewtondiff}, for arbitrarily fixed $\epsilon>0$ we can find $\delta>0$ such that $\B_\delta(\ub)\subset U$ and
  \[\norm{\nabla e_\lambda(u')-\nabla e_\lambda(\bar u)-A(u'-\bar u)}\leq \epsilon\norm{u'-\ub}, u'\in \B_\delta(\ub),\ A\in\co\onabla \nabla e_\lambda(u').\]
   For fixed $u\in\B_\delta(\ub)$, by Rademacher's theorem, the set $[0,1]\setminus D_{\psi_u}$ has Lebesgue measure $0$ and for every $t\in D_{\psi_u}$  we have $\psi_u'(t)=\skalp{A(u-\bar u),u-\bar u}$ for some $A\in \co\onabla \nabla e_\lambda(\bar u+t(u-\bar u))$ by \cite[Theorems 2.3.10, 2.6.6]{Cla83} implying $\vert \psi_u(t)-\psi_u(0)-t\psi_u'(t)\vert\leq \epsilon t\norm{u-\bar u}^2$. Integration by parts yield \[\int_0^1t\psi_u'(t){\rm d}t=1\cdot\psi_u(1)-0\cdot\psi_u(0)-\int_0^1\psi_u(t){\rm d}t=\nabla e_\lambda(u)(u-\ub)-(e_\lambda(u)-e_\lambda(\ub))\]
   and, together with \eqref{EqFuncDiffInt} we obtain that
   \begin{align*}&\lefteqn{\vert \big(e_\lambda(u)-e_\lambda(\ub)-\nabla e_\lambda(\ub)(u-\ub)\big)+\big(e_\lambda(u)-e_\lambda(\ub)-\nabla e_\lambda(u)(u-\ub)\big)\vert}\\
   &=\vert\int_0^1\psi_u(t)-\psi_u(0)-t\psi_u'(t){\rm d}t\vert\leq\int_0^1t\epsilon\norm{u-\ub}^2{\rm d}t=\frac 12 \epsilon\norm{u-\ub}^2
   \end{align*}
implying \eqref{EqTrapRuleMorEnv1} also in this case.

 For every pair $(x,x^*)\in\gph \Ta$ sufficiently close to $(\xb,\xba)$ we have $u:=x+\lambda x^*\in U$, $(x,x^*)=(P_\lambda(u),\nabla e_\lambda(u))$ and $e_\lambda(u)= \varphi(x)+\frac \lambda2\norm{x^*}^2$. Since $\norm{u-\bar u}\leq \norm{x-\bar x}+\lambda\norm{x^*-\xba}$, \eqref{EqTrapRule1} follows from \eqref{EqTrapRuleMorEnv1} by using the identity
  \begin{align*}\lefteqn{e_\lambda(u)-e_\lambda(\bar u)-\frac{\nabla e_\lambda(u)+\nabla e_\lambda(\bar u)}2(u-\bar u)}\\
  &=   \varphi(x)+\frac \lambda 2\norm{x^*}^2-\varphi(\xb)-\frac \lambda 2\norm{\xba}^2-\frac {x^*+\xba}2(x+\lambda x^*-\xb-\lambda\xba)=\varphi(x)-\varphi(\xb)-\frac{x^*+\xba}2(x-\xb).\end{align*}

ad (ii): By \cite[Theorems 4.1, 4.2]{PolRo96b}, the mapping $\Ta$ is strictly proto-differentiable at $\xb$ for $\xba$ if and only if $\nabla e_\lambda$ is strictly differentiable at $\ub$. For arbitrary $\epsilon>0$ consider $\delta>0$ such that $\B_\delta(\ub)\subset U$ and $\norm{\nabla e_\lambda(u')-\nabla e_\lambda(u)-(u'-u)^T\nabla^2 e_\lambda(\ub)}\leq\epsilon \norm{u'-u}$, $u',u\in \B_\delta(\ub)$. Then for arbitrarily fixed $u',u\in\B_\delta(\ub)$ we have
\begin{align*}&\vert e_\lambda(u')-e_\lambda(u)-\nabla e_\lambda(u)(u'-u)-\frac 12(\nabla e_\lambda(u')-\nabla e_\lambda(u))(u'-u)\vert\\
&=\vert\int_0^1(\nabla e_\lambda(u+t(u'-u))-\nabla e_\lambda(u))(u'-u){\rm d}t-\frac 12(\nabla e_\lambda(u')-\nabla e_\lambda(u))(u'-u)\vert\\
&= \vert\int_0^1\big(\nabla e_\lambda(u+t(u'-u))-\nabla e_\lambda(u)-t(u'-u)^T\nabla^2 e_\lambda(\ub)\big)(u'-u){\rm d}t\\
&\qquad-\frac 12\big(\nabla e_\lambda(u')-\nabla e_\lambda(u)-(u'-u)^T\nabla^2 e_\lambda(\ub)\big)(u'-u)\vert\\
&\leq \int_0^1\epsilon t\norm{u'-u}^2{\rm d}t+\frac 12\epsilon\norm{u'-u}^2=\epsilon\norm{u'-u}^2
\end{align*}
implying
\[\lim_{u',u\to \ub, u'\not=0}\frac{e_\lambda(u')-e_\lambda(u)-\frac 12(\nabla e_\lambda(u')+\nabla e_\lambda(u))(u'-u)}{\norm{u'-u}^2}=0.\]
Now similar arguments as we used in (i) verify \eqref{EqTrapRule2}
\end{proof}
\begin{remark}
  We call formulas \eqref{EqTrapRule1}, \eqref{EqTrapRule2} {\em trapezoidal rule} because for differentiable $\varphi$ we have
  \[\varphi(y)-\varphi(x)=\int_0^1\nabla\varphi(x+t(y-x))(y-x){\rm d}t\approx \frac 12(\nabla \varphi(x)+\nabla\varphi(y))(y-x)\]
  by the well-known trapezoidal rule of numerical integration, see, e.g. \cite{Atk89}.
\end{remark}
\begin{remark}
  After submitting this paper, Rockafellar \cite{Ro25} published a manuscript where he introduced the so-called {\em generalized quadratic bundle} ${\rm quad\,} \varphi(\xb,\xba)$ for an lsc function $\varphi:\R^n\to\oR$ which is prox-regular at $\xb$ for $\xba\in \partial f(\xb)$. Since it can be shown that $\Sp_\varphi (\partial\varphi)(\xb,\xba)=\{\gph \partial q\mv q\in {\rm quad\,}\varphi(\xb,\xba)\}$, the list of equivalences in Corollary \ref{CorStrictlyProtSubdiff} can be extended by the statement that ${\rm quad\,} \varphi(\xb,\xba)$ is a singleton.
\end{remark}

\section*{Declarations}

\noindent{\bf Conflict of interest.} The author has no competing interests to declare that are relevant to the content of this article.

\noindent{\bf Data availability. }
Data sharing is not applicable to this article as no datasets have been generated or analysed during the current study.

\end{document}